\documentclass[a4paper, 11pt]{article}
\usepackage{fullpage}

\usepackage{amssymb,amsmath,amsthm,ascmac,array,bm}
\usepackage{graphicx}
\usepackage{color}
\usepackage{epstopdf}
\usepackage{inputenc}
\usepackage{enumerate}
\usepackage{float}
\usepackage{subfigure}
\usepackage{mathrsfs} 
\usepackage{accents}
\usepackage{indentfirst}
\usepackage{varwidth}
\usepackage{parskip}
\usepackage{tikz}
\usetikzlibrary{intersections,calc,arrows.meta}
\usetikzlibrary{quotes}
\usetikzlibrary{knots}

\usepackage{comment}

\usepackage[
colorlinks=true,bookmarks=true,
bookmarksnumbered=true,bookmarkstype=toc,linktocpage=true
]{hyperref}

\usepackage{xcolor}
\usepackage[capitalize,
noabbrev,nosort]{cleveref}
\hypersetup{
	colorlinks=true,       
	linkcolor=blue,          
	citecolor=blue,        
	filecolor=blue,      
	urlcolor=blue,           
}

\newtheorem{theoremcounter}{Theorem Counter}[section]

\theoremstyle{definition}

\newtheorem{ex}[theoremcounter]{Example}

\theoremstyle{plain}
\newtheorem{lem}[theoremcounter]{Lemma}
\newtheorem{prop}[theoremcounter]{Proposition}

\newtheorem{thm}[theoremcounter]{Theorem}

\numberwithin{equation}{section}

\allowdisplaybreaks

\newcommand{\ubreve}[1]{\underaccent{\breve}{#1}}

\newcommand{\mbbz}{\mathbb{Z}}

\usepackage{amsthm}
\theoremstyle{plain}

\newtheorem*{thmA}{Theorem A}
\newtheorem*{thmB}{Theorem B}

\begin{document}

   \title{\Large \bf Arithmetic Dijkgraaf--Witten invariants for real quadratic fields, quadratic residue graphs, and density formulas}
    \author{Yuqi Deng, Riku Kurimaru, and Toshiki Matsusaka}


    

\date{}    
\maketitle
  
\begin{abstract}

We compute Hirano's formula for the mod 2 arithmetic Dijkgraaf--Witten invariant ${Z}_k$ for the ring of integers of the quadratic field $k=\mathbb{Q}(\sqrt{p_1\cdots p_r})$, where ${p_i}$'s are distinct prime numbers with $p_i \equiv 1 \pmod{4}$, and give a simple formula for $Z_k$ in terms of the graph obtained from quadratic residues among $p_1,\cdots, p_r$. Our result answers the question posed by Ken Ono. We also give a density formula for mod 2 arithmetic Dijkgraaf--Witten invariants.
\end{abstract}

\vspace{\baselineskip}
\begin{center}
\title\textbf{Introduction}
\end{center}

In \cite{M}, Minhyong Kim initiated to study arithmetic Chern--Simons theory for number rings, based on Dijkgraaf--Witten theory for 3-manifolds (\cite{Dijkgraaf}) and the analogies between number rings and 3-manifolds, primes and knots in arithmetic topology (\cite{Morishita}). Kim's theory is concerned with totally imaginary number fields, since it employs some results on \'{e}tale cohomology groups of the integer rings of totally imaginary number fields (\cite{Mazur}). Later Kim's construction was extended for any number field which may have real primes by Hirano~\cite{hirano} and by Lee--Park~\cite{Lee}, and Hirano~\cite{hirano} introduced the mod $n$ arithmetic Dijkgraaf--Witten invariant for any number ring containing a primitive $n$-th root of unity. As an interesting example, Hirano computed the mod 2 arithmetic Dijkgraaf--Witten invariant $Z_k$ for the real quadratic field $k=\mathbb{Q}(\sqrt{p_1\cdots p_r})$, where ${p_i}$'s are distinct prime numbers with $p_i \equiv 1 \pmod{4}$, and showed the following formula expressing $Z_k$ in terms of the quadratic residue symbols among ${p_i}$'s:
\begin{align}\label{01}
	Z_k = \frac{1}{2} \sum_{\rho \in \mathrm{Hom} (T,\mathbb{Z}/2\mathbb{Z})} \left(\prod_{i<j}\left(\frac{p_i}{p_j} \right)^{\rho(b_{ij})} \right),
\end{align}
where $T := \{(x_1,\cdots,x_r) \in (\mathbb{Z}/2\mathbb{Z})^r \mid \sum_{i=1}^{r} x_i = 0\}$, $b_{ij} = (0, \cdots, 0, \overset{\ubreve{i}}{1}, 0, \cdots, 0, \overset{\ubreve{j}}{1}, 0, \cdots, 0)$, and $\mathrm{Hom}(T,\mathbb{Z}/2\mathbb{Z})$ is the set of homomorphisms $T\rightarrow{\mathbb{Z}/2\mathbb{Z}}$.

At the conference of ``Low dimensional topology and number theory XI" held at Osaka University in March of 2019, Ken Ono asked us if the right-hand side of Hirano's formula \eqref{01} could be simplified and suggested computing numerically many examples for ${p_i}$'s in order to find such a simple formula.

In this paper, we answer Ono's question. In fact, after many numerical computer calculations, we found and proved the following simple formula for $Z_k$ using a certain graph $G(S)$, called the quadratic residue graph associated with the set $S = \{p_1, \dots, p_r\}$, which is determined by the quadratic residue symbols $\left(\frac{p_i}{p_j}\right)$, (see \cref{s3} for the definition of $G(S)$).

\begin{thmA} [Theorem 3.2 below]
	We have
	\begin{align}
		Z_k = \begin{cases}
			2^{r-2}&\text{if any connected component of $G(S)$ is a circuit,}\\
			0&\text{if otherwise.}
		\end{cases}
	\end{align}
\end{thmA}

We note that the graph $G(S)$ is an arithmetic analog of the linking diagram in link theory, and the idea to use the graph $G(S)$ was suggested by the analogy between primes and knots in arithmetic topology. In fact, we can show a similar formula for the topological mod 2 Dijkgraaf--Witten invariants for double covers of $S^3$ using the linking diagram of branched knots. 

Let $\mathfrak{G}_r$ be the set of all graphs with the vertex set $\{1,2,\dots, r\}$ and $\mathfrak{C}_r$ be the subset of $\mathfrak{G}_r$ consisting of graphs whose connected components are all circuits.  We compute the density of $\mathfrak{C}_r$ in $\mathfrak{G}_r$ (\cref{C/G-density} below). Let $\mathcal{P}$ denote the set of all prime numbers and $P_r(x) := \{ \{p_1, \cdots, p_r\} \subset \mathcal{P} \cap [1,x] \mid p_i \equiv 1 \pmod{4}, p_i \neq p_j\ (i \neq j) \}$. Then we show the density formula for mod $2$ arithmetic Dijkgraaf--Witten invariants among all real quadratic field $k=\mathbb{Q}(\sqrt{p_1\cdots p_r})$, where $p_1,\cdots,p_r$ are distinct prime numbers with $p_i \equiv 1 \pmod{4}$.

\begin{thmB} [\cref{density-theorem} below]
	We have
	\[
		\lim_{x\to\infty} \frac{\#\left\{\{p_1,\cdots,p_r\}\in P_r(x) \mid Z_k=2^{r-2}\right\}}{\#P_r(x)}=\frac{\#\mathfrak{C}_r}{\#\mathfrak{G}_r}=\frac{1}{2^{r-1}}.
	\]
\end{thmB}

The contents of this paper are organized as follows. In \cref{s1}, we recall the mod $n$ arithmetic Dijkgraaf--Witten invariant $Z_k$ for the ring of integers of a number field $k$ containing a primitive $n$-th root of unity, and Hirano's formula for the mod 2 arithmetic Dijkgraaf--Witten invariant $Z_k$ for the real quadratic field $k=\mathbb{Q}(\sqrt{p_1\cdots p_r})$, where ${p_i}$'s are distinct prime numbers with $p_i \equiv 1 \pmod{4}$. In \cref{s2}, we recall some basic notions on graphs. In \cref{s3}, we introduce the quadratic residue graph $G(S)$ associated with the set $S = \{ p_1, \dots , p_r\}$, and state our main theorem (cf.~\cref{maint} below), which is a simple formula expressing $Z_k$, $k = \mathbb{Q}(\sqrt{p_1 \cdots p_r})$, in terms of $G(S)$. In \cref{s4}, we give a proof of the main theorem. In \cref{s5}, we give a topological counterpart of our result for the mod 2 topological Dijkgraaf--Witten invariant for a double cover of the 3-sphere $S^3$ branched over a link. In \cref{s6}, we calculate the density of graphs with $r$ vertices whose connected components are all circuits. In \cref{s7}, we propose a density formula for mod 2 arithmetic Dijkgraaf--Witten invariants, based on the properties of quadratic residue graphs, and we find the density of graphs whose connected components are all circuits and the density of mod 2 arithmetic Dijkgraaf--Witten invariants are equal.


\subsection*{Acknowledgement}

The authors thank Professor Ken Ono for suggesting the question which motivated our study. Y. Deng and R. Kurimaru also express their gratitude to our supervisor, Professor Masanori Morishita, for his valuable advice. Additionally, the authors extend their thanks to Yuta Suzuki for introducing us to \cref{HB}. The third author was supported by JSPS KAKENHI Grant Numbers JP20K14292 and JP21K18141.

\section{Hirano's formulas for the mod 2 arithmetic Dijkgraaf--Witten invariants} \label{s1}

 In this section, we recall the definition of the mod $n$ arithmetic Dijkgraaf--Witten invariant for a number ring, along with Hirano's formula and Ono's question for the mod 2 arithmetic Dijkgraaf--Witten invariant for real quadratic fields $\mathbb{Q}(\sqrt{p_1\cdots p_r})$, $p_i \equiv 1 \pmod{4}$. 
 
Let $k$ be a number field  of finite degree over $\mathbb{Q}$. We fix a primitive $n$-th root of unity $\zeta_n$ and assume that $k$ contains $\zeta_n$. Let $\mathcal{O}_k$ be the ring of integers of $k$ and $S_{k}^{\infty}$ the set of infinite primes of $k$. We set $\overline{X}_k = {\rm Spec}(\mathcal{O}_k)\sqcup{S_{k}^{\infty}}$. 
Let $\pi_1(\overline{X}_k)$ be the modified \'{e}tale fundamental group of $\overline{X}_k$ defined by considering the Artin--Verdier topology over $\overline{X}_k$, which takes the real primes of $k$ into account (cf.~\cite[\S 2.1]{hirano}). It is the Galois group of the maximal extension over $k$ unramified at all finite and infinite primes.

 Let $A$ be a finite group with discrete topology, and $c$ be a fixed cohomology class $c \in {H^3}(A,\mbbz/n\mbbz)$. Let $\mathrm{Hom}_c(\pi_1(\overline{X}_k),A)$ be the set of continuous homomorphisms from $\pi_1(\overline{X}_k)$ to $A$. For $\rho \in \mathrm{Hom}_c(\pi_1(\overline{X}_k),A)$, we define the mod $n$ \emph{arithmetic Chern--Simons invariant} $CS_c(\rho)$ by the image of $c$ under the composition
\[
	H^3(A,\mbbz/n\mbbz) \overset{\rho^*}{\to} {H^3(\pi_1(\overline{X}_k),\mbbz/n\mbbz)} \overset{j}{\to} {H^3(\overline{X}_k,\mbbz/n\mbbz)} \cong {\mbbz/n\mbbz},
\]
where the cohomology group of $\overline{X}_k$ is the modified \'{e}tale cohomology group defined in the Artin--Verdier topology, and $j$ is the edge homomorphisms in the modified Hochschild--Serre spectral sequence 
\[
	H^p(\pi_1(\overline{X}_k),H^q(\widetilde{{\overline{X}}}_k,\mbbz/n\mbbz))\Rightarrow{H^{p+q}(\overline{X}_k,\mbbz/n\mbbz)},
\]
where $\widetilde{{\overline{X}}}_k = \underleftarrow{\lim}Y_i$, $Y_i$ running over a finite Galois covering of $\overline{X}_k$ (cf.~\cite[\S 2.2]{hirano}). Note that the isomorphism $H^3(\overline{X}_k,\mbbz/n\mbbz)\cong{\mbbz/n\mbbz}$ depends on the choice of $\zeta_n$.

We then define the mod $n$ {\it arithmetic Dijkgraaf--Witten invariant} $Z_c(\overline{X}_k)$ of $\overline{X}_k$ by
\[
	Z_c(\overline{X}_k) := \frac{1}{\# A}\sum_{\rho \in \mathrm{Hom}_c(\pi_1(\overline{X}_k),A)} {{\zeta_n}^{CS_c(\rho)}}.
\]
This definition differs from~\cite{hirano} by one factor $1/\# A$.
In the following,  we write simply $Z_k$ for $Z_c(\overline{X}_k)$. 

Now let $k$ be the quadratic field $\mathbb{Q}(\sqrt{p_1\cdots p_r})$, where ${p_i}$'s are distinct prime numbers with $p_i \equiv 1 \pmod{4}$. We consider the case that $n=2$ and $A=\mbbz/2\mbbz$. Let $c$ be the unique non-trivial class in $H^3(A,\mbbz/2\mbbz)$. We set $T:= \{(x_1,\cdots,x_r) \in {(\mathbb{Z}/2\mathbb{Z})^r \mid \sum_{i=1}^{r} x_i = 0}\}$ and  $b_{ij} :=(0, \cdots, 0, \overset{\ubreve{i}}{1}, 0, \cdots, 0, \overset{\ubreve{j}}{1}, 0, \cdots, 0) \in T$. Let $\mathrm{Hom}(T,\mathbb{Z}/2\mathbb{Z})$ be the abelian group of homomorphisms from $T$ to $\mathbb{Z}/2\mathbb{Z}$. Then Hirano showed the following formula for the mod $2$ Dijkgraaf--Witten invariants $Z_k$.

\begin{thm}[{\cite[Corollary 4.2.4]{hirano}}]\label{Hirano-formula}
	We have
	\[
		Z_k=\frac{1}{2}\sum_{\rho\in{{\rm Hom}(T,\mathbb{Z}/2\mathbb{Z})}}{\left(\prod_{i<j}{\left({\frac{p_i}{p_j}}\right)}^{\rho(b_{ij})}\right)}.
	\]
\end{thm}


Ken Ono asked us the question, ``Can we simplify the right-hand side of the formula in \cref{Hirano-formula}", and he suggested computing numerically the right-hand side for many examples of $p_i$'s.

\section{Preliminaries on graphs} \label{s2}

In this section, we recall some basic notions on graphs, which will be used in the subsequent sections.

A \emph{graph} $G$ consists of two sets $V = V(G)$ and $E =E(G)$, where $V$ is the set of \emph{vertices} and $E$ is the set of \emph{edges}. The graph is denoted by $G=(V,E)$. 
The set $E$ is regarded as a subset of the power set $\mathfrak{P}(V)$ consisting of $2$-sets, $E \subset \{e \in \mathfrak{P}(V) \mid \# e = 2\}$.
We denote by $e_{ij}$ ($= \{v_i, v_j\}$) the edge joining the vertices $v_i$ and $v_j$. Here $e_{ii}$ is not considered as an edge.

We say that $G'=(V',E')$ is a \emph{subgraph} of $G=(V,E)$ if $V'\subset{V}$, $E'\subset{E}$, denoted by $G'\subset{G}$.
Two graphs $G$ and $H$ are \emph{isomorphic} if there is a bijection $f$ between the vertex sets $V(G)$ and $V(H)$ such that any two vertices $u$ and $v$ of $G$ are adjacent in $G$ if and only if $f(u)$ and $f(v)$ are adjacent in $H$.

The \emph{degree} of a vertex $v\in{V}$, denoted by $\deg(v)$, is the number of vertices which are adjacent to $v$. We say that $v$ is an \emph{even} (resp.~\emph{odd}) vertex if $\deg(v)$ is even (resp.~odd), ($\deg(v)=0$ included). We call a graph $G$ an even (resp.~odd) graph if all $v\in{V}$ are even (resp.~odd). 

A  \emph{path} is a graph $P=(V,E)$ of the form
\[
	V=\{v_0,v_1,\cdots,v_l\}, E=\{e_{01}, e_{12},e_{23},\cdots,e_{l-1,l}\},
\]
where vertices $v_0,v_1,\dots, v_l$ are distinct each other. This path $P$ is denoted by $v_0v_1\cdots{v_l}$.
A graph is \emph{connected} if there exists a path between any two vertices in the graph. A graph with only one vertex is regarded as a connected graph. A maximal connected subgraph of a graph $G$ is called a \emph{connected component} of $G$.

A sequence $v_{i_1}e_{i_{1}i_{2}}v_{i_2}\cdots v_{i_j} \cdots v_{i_{r-1}}e_{i_{r-1}i_r}v_{i_r}$ is called a \emph{trail} in a graph, when vertices $v_{i_1}, \dots, v_{i_r}$ and edges $e_{i_1 i_2}, \dots, e_{i_{r-1} i_r}$ appear alternately and every edge appear exactly once, as in \cref{Ex-trails} below. The path is an example of the trail. The first vertex $v_{i_1}$, the last vertex $v_{i_r}$, and the other vertices $v_{i_j}$ of the trail are called the \emph{starting vertex}, the \emph{terminal vertex}, and \emph{passing vertices}, respectively. When the starting vertex and the terminal vertex of a trail coincide, the trail is called a \emph{circuit}. (In some references, it is called an \emph{Euler tour.}) A graph $G$ consisting of a single vertex is considered to be a circuit.

\begin{figure}[H]
 \begin{tikzpicture}[>=Stealth]
   \coordinate (a) at (-2.5,1);
  \coordinate (b) at (0,-0.5);
    \coordinate (c) at (0,1);
        \coordinate (d) at (2.5,-0.5);
  \node  at (a)[anchor=225] {\small{\rm$v_1$}} ;
\node at (b)[anchor=90] {\small{\rm$v_2$}} ;
\node at (c) [anchor=-90]{\small{\rm$v_3$}} ;
\node at (d) [anchor=90]{\small{\rm$v_4$}} ;
   \node  at (-1.5,0.1){\small{\rm$e_{12}$}} ;
     \node  at (0.3,0.1){\small{\rm$e_{23}$}} ;
       \node  at (1.5,0.6){\small{\rm$e_{34}$}} ;
         \fill (a) circle (1.3pt);
                \fill (b) circle (1.3pt);
                       \fill (c) circle (1.3pt);
                       \fill (d) circle (1.3pt);
                \draw [line width=1pt] (a)--(b)--(c)--(d);            
  \begin{scope}[xshift=5.5cm]
    \coordinate (a) at (-2,-0.5);
  \coordinate (b) at (0,-0.5);
    \coordinate (c) at (2.5,-0.5);
       \coordinate (d) at (1.25,1.3);
  \node  at (a)[anchor=-90] {\small{\rm$v_1$}} ;
\node at (b)[anchor=-70] {\small{\rm$v_2$}} ;
\node at (c) [anchor=-180]{\small{\rm$v_3$}} ;
\node at (d) [anchor=-90]{\small{\rm$v_4$}} ;
         \fill (a) circle (1.3pt);
                \fill (b) circle (1.3pt);
                       \fill (c) circle (1.3pt);
                             \fill (d) circle (1.3pt);
                \draw [line width=1pt] (a)--(b)--(c)--(d)--(b);
                   \node  at (-1,-0.7){\small{\rm$e_{12}$}} ;
     \node  at (0.4,0.6){\small{\rm$e_{24}$}} ;
       \node  at (2.1,0.6){\small{\rm$e_{34}$}} ;
              \node  at (1.25,-0.7){\small{\rm$e_{23}$}} ;
              \end{scope} 
       \begin{scope}[xshift=11cm]
    \coordinate (a) at (-0.7,1.5);
  \coordinate (b) at (0.7,1.5);
    \coordinate (c) at (-1.5,-0.5);
        \coordinate (d) at (1.5,-0.5);
  \node  at (a)[anchor=-90] {\small{\rm$v_1$}} ;
\node at (b)[anchor=-90] {\small{\rm$v_2$}} ;
\node at (c) [anchor=0]{\small{\rm$v_3$}} ;
\node at (d) [anchor=-180]{\small{\rm$v_4$}} ;
         \fill (a) circle (1.3pt);
                \fill (b) circle (1.3pt);
                       \fill (c) circle (1.3pt);
                          \fill (d) circle (1.3pt);
                \draw [line width=1pt] (a)--(b)--(c)--(d)--(a);
 \node  at (0,1.7){\small{\rm$e_{12}$}} ;
\node  at (-0.8,0.6){\small{\rm$e_{23}$}} ;
\node  at (0.8,0.6){\small{\rm$e_{14}$}} ;
\node  at (0,-0.7){\small{\rm$e_{34}$}} ;
\end{scope}          
 \end{tikzpicture}
 \caption{Examples for trails. The right one is a circuit.}
 \label{Ex-trails}
 \end{figure}
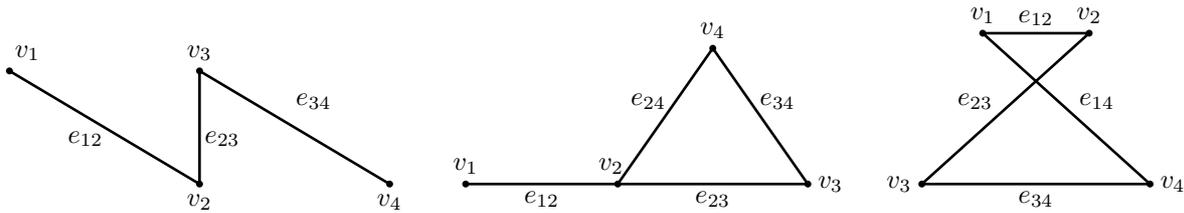
 
The following is Euler's famous result concerning the classification of graphs that we will use later (cf.~\cite[Chapter I, Theorem 12]{boll}).
 
\begin{thm}[Euler 1736]\label{G-even-graph}
	For a graph $G$, the following conditions are equivalent.
	\begin{enumerate}
		\item[$(1)$] Any connected component of $G$ is a circuit.
		\item[$(2)$] $G$ is an even graph.
	\end{enumerate}
\end{thm}
 




\section{Main theorem} \label{s3}

In this section, we introduce the quadratic residue graph associated to a finite set of prime numbers $p_1,\cdots, p_r \equiv 1 \pmod{4}$, 
and we present the main theorem in this paper, which computes Hirano's formula explicitly.

Let $S=\{p_1,p_2,\cdots,p_r\}$ be a finite set of distinct prime numbers, where $p_i \equiv{1} \pmod{4}$ $ (1\leq{i}\leq{r})$. We define the \emph{quadratic residue graph} $G(S)$ associated to $S$ by 
\begin{itemize}
	\item $V = V(G(S)) = \{p_1, \dots, p_r\}$.
	\item $E = E(G(S)) = \{e = \{p_i, p_j\} \in \mathfrak{P}(V) \mid \left(\frac{p_i}{p_j}\right) = -1\}$.
\end{itemize}

We can also illustrate the graph $G(S)$ as follows. We set primes $p_1,p_2, \cdots, p_r$ in order so that $p_1<p_2<\cdots<p_r$. Then, we put the vertices of $G(S)$ evenly on a unit circle counterclockwise starting at the point $(1, 0)$. 
Namely, the vertex  $(\cos\frac{2\pi(i-1)}{r},\sin\frac{2\pi(i-1)}{r})$ corresponds to the prime $p_i$, and we  denote the vertices by ${p_i}$'s. Then, two vertices $p_i$ and $p_j$ are linked (adjacent) if and only if $\left(\frac{p_i}{p_j}\right) = -1$ $(i \neq j)$. Since $p_i \equiv{1} \pmod{4}$, the graph $G(S)$ is uniquely well-defined by the quadratic reciprocity.

\begin{ex}
\begin{enumerate}[\rm (1)]
\item $r=3$, $S=\{5,13,37\}$.
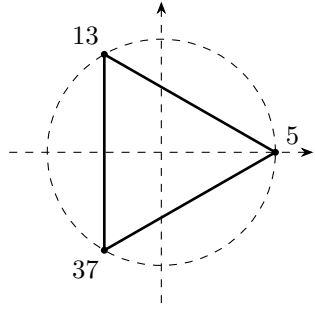
\begin{center}
 \begin{tikzpicture}[>=Stealth]
   \coordinate (a) at (1.5,0);
  \coordinate (b) at (-0.75,1.299);
    \coordinate (c) at (-0.75,-1.299);
  \node  at (a)[anchor=225] {\small{\rm5}} ;
\node at (b)[anchor=-45] {\small{\rm13}} ;
\node at (c) [anchor=45]{\small{\rm37}} ;
    \draw[dashed] (0,0) circle(1.5);
     \draw[dashed,->](0,-2)--(0,2);
          \draw[dashed,->](-2,0)--(2,0);
          
         \fill (a) circle (1.3pt);
                \fill (b) circle (1.3pt);
                       \fill (c) circle (1.3pt);
                \draw [line width=1pt] (a)--(b)--(c)--(a);
 \end{tikzpicture}
 \end{center}
\item $r=4, S=\{5, 13, 37,41\}$.
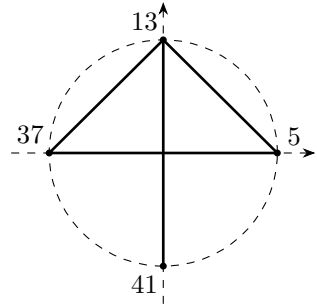
\begin{center}
 \begin{tikzpicture}[>=Stealth]
 
   \coordinate (a) at (1.5,0);
  \coordinate (b) at (0,1.5);
    \coordinate (c) at (-1.5,0);
  \coordinate (d) at (0,-1.5);

  \node  at (a)[anchor=225] {\small{\rm5}} ;
\node at (b)[anchor=-45] {\small{\rm13}} ;
\node at (c) [anchor=-45]{\small{\rm37}} ;
\node at (d) [anchor=45]{\small{\rm41}} ;

    \draw[dashed] (0,0) circle(1.5);
     \draw[dashed,->](0,-2)--(0,2);
          \draw[dashed,->](-2,0)--(2,0);
          
         \fill (a) circle (1.3pt);
                \fill (b) circle (1.3pt);
                       \fill (c) circle (1.3pt);
                              \fill (d) circle (1.3pt);
                              \draw [line width=1pt] (a)--(b)--(c)--(a);
                                \draw  [line width=1pt](b)--(d);

 \end{tikzpicture}
 \end{center}
 \end{enumerate}
\end{ex}

Many numerical examples are given in \cref{appendix}, from which we find the following main theorem.

\begin{thm}\label{maint} 
	We have
	\begin{align*}
		Z_k &= \frac{1}{2} \sum_{\rho \in \mathrm{Hom}(T,\mathbb{Z}/2\mathbb{Z})} \left(\prod_{i<j}{\left({\frac{p_i}{p_j}}\right)}^{\rho(b_{ij})}\right)\\ 
		&= \begin{cases}
			2^{r-2} &\text{if any connected component of $G(S)$ is a circuit,}\\
			0 &\text{if otherwise.}
		\end{cases}
	\end{align*}
\end{thm}

\begin{ex}
	Let $S=\{5, 29, 37,73\}$ so that $\left(\frac{5}{37}\right) = \left(\frac{5}{73}\right) = \left(\frac{29}{37}\right) = \left(\frac{29}{73}\right)  =-1$, $\left(\frac{5}{29}\right) =\left(\frac{37}{73}\right) = 1$. Then the quadratic residue graph $G(S)$ is given by the following figure. Let $k$ := $\mathbb{Q}(\sqrt{5\cdot29\cdot37\cdot73})$ = $\mathbb{Q}(\sqrt{391645})$. 
By \cref{maint}, we have $Z_k = 2^2 = 4$. 
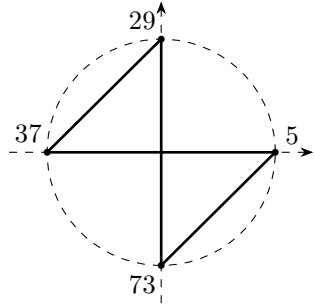
\begin{center}
 \begin{tikzpicture}[>=Stealth]
 
   \coordinate (a) at (1.5,0);
  \coordinate (b) at (0,1.5);
    \coordinate (c) at (-1.5,0);
  \coordinate (d) at (0,-1.5);

  \node  at (a)[anchor=225] {\small{\rm5}} ;
\node at (b)[anchor=-45] {\small{\rm29}} ;
\node at (c) [anchor=-45]{\small{\rm37}} ;
\node at (d) [anchor=45]{\small{\rm73}} ;

    \draw[dashed] (0,0) circle(1.5);
     \draw[dashed,->](0,-2)--(0,2);
          \draw[dashed,->](-2,0)--(2,0);
          
         \fill (a) circle (1.3pt);
                \fill (b) circle (1.3pt);
                       \fill (c) circle (1.3pt);
                              \fill (d) circle (1.3pt);
                              \draw [line width=1pt] (a)--(c)--(b)--(d)--(a);

 \end{tikzpicture}
 \end{center}
\end{ex}

\section{Proof of the main theorem} \label{s4}

We give a proof of \cref{maint} by using the orthogonality of characters of a finite abelian group. 
Throughout this section, we use the notations given in the previous sections.
 
 \begin{lem}\label{bij=0}
 	Let $G(S)$ be the quadratic residue graph associated to $S=\{p_1,\cdots, p_r\}$, $p_i \equiv{1} \pmod{4}$ $(1\leq{i}\leq{r})$. Then, $G(S)$ is an even graph if and only if $\sum_{\{p_i, p_j\} \in E(G(S))} b_{ij} = \mathbf{0}$.
 \end{lem}
 
\begin{proof}
	By the definition of $b_{ij} \in T$, we easily see that
	\[
		\sum_{\{p_i, p_j\} \in E(G(S))} b_{ij} = \left( \deg(p_i) \bmod{2}\right)_{1 \leq i \leq r}.
	\]
	Therefore, the right-hand side equals $\mathbf{0}$ if and only if $G(S)$ is an even graph.
\end{proof}



Next we recall that the following orthogonality of characters of a finite abelian group, (see, for example, \cite[Section 1.12]{ono}.) 

\begin{lem}\label{orthogonal}
	Let $A$ be a finite abelian group of order $n$. Let $\chi: A \to \mathbb{C}^{\times}$ be a homomorphism (character of $A$). Then we have
\[
	\sum_{a\in{A}}{\chi(a)} = \begin{cases}
		n &\text{if } \chi = \varepsilon,\\
		0&\text{if } \chi \neq \varepsilon,
	\end{cases}
\]
where $\varepsilon$ is the identity character defined by $\varepsilon(a)$ = $1$ for any $a\in{A}$.
\end{lem}

Now we define the map $\varphi: \mathrm{Hom}(T,\mathbb{Z}/2\mathbb{Z}) \to \mathbb{C}^{\times}$ by
\[
	\varphi(\rho):=\prod_{i<j}{{\left(\frac{p_i}{p_j}\right)}^{\rho(b_{ij})}}\in{\{\pm 1\}}.
\]
Since we have, for $\rho_1, \rho_2 \in \mathrm{Hom}(T,\mathbb{Z}/2\mathbb{Z})$,
\begin{align*}
	\varphi(\rho_1+\rho_2) 
		=\prod_{i<j}{\left(\frac{p_i}{p_j}\right)^{\rho_1(b_{ij})+\rho_2(b_{ij})}}
		=\varphi(\rho_1)\varphi(\rho_2),
\end{align*}
$\varphi$ is a homomorphism. By \cref{orthogonal}, we have
\[
	\sum_{\rho \in \mathrm{Hom}(T,\mathbb{Z}/2\mathbb{Z})}{\varphi(\rho)} = \begin{cases}
		2^{r-1} &\text{if } \varphi=\varepsilon,\\
		0 &\text{if } \varphi\ne\varepsilon,
	\end{cases}
\]
since the order of $\mathrm{Hom}(T,\mathbb{Z}/2\mathbb{Z})$ is $2^{r-1}$.

Therefore, proving \cref{maint} is equivalent to showing the following theorem.

\begin{thm}
	For a quadratic residue graph $G(S)$, the following conditions are equivalent. 
	\begin{enumerate}[(1)]
		\item[$(1)$] $\varphi=\varepsilon$.
		\item[$(2)$] Any connected component of $G(S)$ is a circuit.
	\end{enumerate}
\end{thm}

\begin{proof}
Suppose that any connected component of $G(S)$ is a circuit.
According to \cref{bij=0}, we know that $\sum_{\{p_i, p_j\} \in E(G(S))} b_{ij} = \mathbf{0}$.
For any $\rho \in \mathrm{Hom}(T,\mathbb{Z}/2\mathbb{Z})$,
\begin{align}\label{phi-rho-0}
	\varphi(\rho) &=\prod_{i<j}{\left(\frac{p_i}{p_j}\right)^{\rho(b_{ij})}}
		= \prod_{\{p_i,p_j\} \in E(G(S))} (-1)^{\rho(b_{ij})}
		= (-1)^{\sum_{\{p_i, p_j\} \in E(G(S))} \rho(b_{ij})}.
\end{align}
Since the exponent equals $\rho(\mathbf{0}) = 0$, we have $\varphi={\varepsilon}$.
 
On the other hand, if $\varphi = \varepsilon$, then \eqref{phi-rho-0} implies that $\rho \left(\sum_{\{p_i, p_j\} \in E(G(S))} b_{ij} \right) = 0$ for any $\rho \in \mathrm{Hom}(T, \mathbb{Z}/2\mathbb{Z})$. Therefore, we have $\sum_{\{p_i, p_j\} \in E(G(S))} b_{ij} = \mathbf{0}$. By \cref{bij=0} again, the graph $G(S)$ is even and any connected component of $G(S)$ is a circuit.
\end{proof}

\section{Mod 2 Dijkgraaf--Witten invariants of double covers of the 3-sphere} \label{s5}

In this section, we give a topological counterpart of \cref{maint} for the mod 2 topological Dijkgraaf--Witten invariant for a double cover of the 3-sphere $S^3$ branched over a link (cf.~\cite[\S 5]{hirano}).
Let $M$ be a connected, oriented, and closed 3-manifold, and let $n$ be an integer with $n \ge 2$. Let $A$ be a finite group and let $c \in {H}^{3}(A, \mathbb{Z} / n \mathbb{Z})$. We have ${H}_{3}(M, \mathbb{Z} / n \mathbb{Z}) \cong \mathbb{Z} / n \mathbb{Z}$ and we denote by $[M] \in {H}_{3}(M, \mathbb{Z} / n \mathbb{Z})$ the fundamental homology class of $M$.
 
For $\rho \in \mathrm{Hom}\left(\pi_{1}(M), A\right)$, we define the mod $n$ {\it arithmetic Chern--Simons invariant} $CS_c(\rho)$ by the image of $c$ under the composition
\[
	{H}^{3}(A, \mathbb{Z} / n \mathbb{Z}) \stackrel{\rho^{*}}{\longrightarrow} H^{3}\left(\pi_{1}(M), \mathbb{Z} / n \mathbb{Z}\right) \stackrel{j_{3}}{\longrightarrow} {H}^{3}(M, \mathbb{Z} / n \mathbb{Z}) \stackrel{\langle \cdot, [M] \rangle}{\longrightarrow} \mathbb{Z} / n \mathbb{Z},
\]
where $j_3$ is the edge homomorphisms  in the modified Hochschild--Serre spectral sequence
\[
	{H}^{p}(\pi_{1}(M), {H}^{q}(\widetilde{M}, \mathbb{Z} / n \mathbb{Z})) \Rightarrow {H}^{p+q}(M, \mathbb{Z} / n \mathbb{Z}).
\]
We then define the mod $n$ {\it Dijkgraaf--Witten invariant} of $M$ associated to $c$ by
\[
	Z_{c}(M) = \frac{1}{\# A} \sum_{\rho \in \mathrm{Hom}(\pi_{1}(M), A)} \exp \left(\frac{2 \pi i}{n} CS_{c}(\rho)\right).
\]
We write simply $Z_M$ for $Z_{c}(M)$.

Now consider the case where $A=\mathbb{Z} / 2 \mathbb{Z}$ and $c\in {H}^{3}(A, \mathbb{Z} / 2 \mathbb{Z})$ is the unique non-trivial class. Let $\mathcal{L}=\mathcal{K}_{1} \cup \mathcal{K}_{2} \cup \cdots \cup \mathcal{K}_{r}$ be a tame link in the  3-sphere $S^{3}$ and let $h: M \rightarrow S^{3}$ be the double covering ramified over $\mathcal{L}$ obtained by the unramified covering $Y \rightarrow X:=S^{3} \setminus \mathcal{L}$ corresponding to the kernel of the surjective homomorphism ${H}_{1}(X) \rightarrow \mathbb{Z} / 2 \mathbb{Z}$ that maps any meridian of $\mathcal{K}_{i}$ to $1 \in \mathbb{Z} / 2 \mathbb{Z}$.

Then Hirano showed the following formula for the mod 2 Dijkgraaf--Witten invariants of double covers of the 3-sphere.

\begin{thm}[{\cite[Corollary 5.3.2]{hirano}}]
	We have
	\[
		Z_M = \frac{1}{2} \sum_{\rho \in \mathrm{Hom} (T, \mathbb{Z} / 2 \mathbb{Z})} \exp \left(\pi i \sum_{i<j} \rho\left(b_{i j}\right) \mathrm{lk}\left(\mathcal{K}_{i}, \mathcal{K}_{j}\right) \bmod 2\right),
	\]
	where $\mathrm{lk}(\cdot,\cdot) \bmod 2$ denotes the mod $2$ linking number.
\end{thm}

Then, we define the \emph{linking graph} $D_{\mathcal{L}}$ associated to $\mathcal{L}$ as follows. We put $\mathcal{K}_{1}, \mathcal{K}_{2},  \dots, \mathcal{K}_{r}$ clock-wisely as vertices of the regular $r$-polygon. Two vertices $\mathcal{K}_i$ and $\mathcal{K}_j$ are linked (adjacent) if $\mathrm{lk}\left(\mathcal{K}_{i}, \mathcal{K}_{j}\right) \equiv 1 \bmod 2$.  Since $\mathrm{lk}\left(\mathcal{K}_{i}, \mathcal{K}_{j}\right)=\mathrm{lk}\left(\mathcal{K}_{j}, \mathcal{K}_{i}\right)$, $D_{\mathcal{L}}$ is well-defined.

Then by the same method as in the arithmetic case, we have,
\begin{thm}\label{main-top}
	We have
	\[
		Z_M = \begin{cases}
			2^{r-2} &\text{if any connected component of $D_{\mathcal{L}}$ is a circuit,}\\
			0 &\text{if otherwise.}
		\end{cases}
	\]
\end{thm}

\begin{ex}
	Let $\mathcal{L} = \mathcal{K}_{1} \cup \mathcal{K}_{2} \cup \mathcal{K}_{3} \cup \mathcal{K}_{4}$ be the following link (left figure) in $S^{3}$ so that the linking graph $D_{\mathcal{L}}$ is given by the right figure. Let $M$ be the double covering of $S^{3}$ ramified along $\mathcal{L}$.
By \cref{main-top}, we have $Z_M = 2^2 = 4$. 

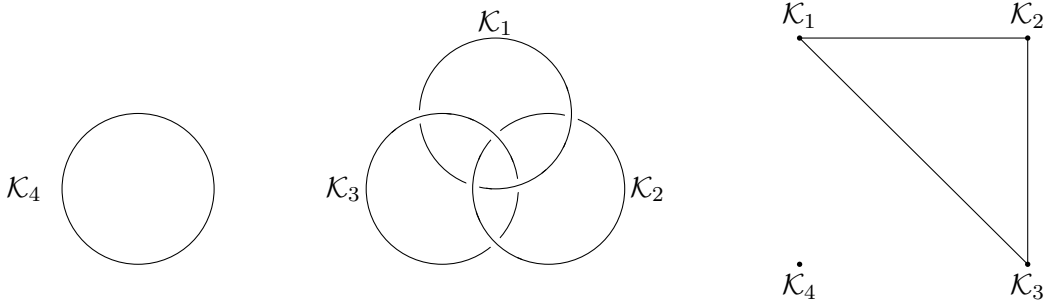
\begin{tikzpicture}[scale=1.0, >=Stealth]
\begin{knot}[
clip width=3.5,
flip crossing=1,
flip crossing=3,
flip crossing=6
]
\strand (2,2) circle [radius=1];
\strand (1.3,1) circle [radius=1];
\strand (2.7,1) circle [radius=1];
\end{knot}
\node at(2,3.2) {$\mathcal{K}_{1}$};
\node at(0,1) {$\mathcal{K}_{3}$};
\node at(4,1) {$\mathcal{K}_{2}$};
\draw(-2.7,1)circle(1);
\node at(-4.2,1) {$\mathcal{K}_{4}$};

\begin{scope}[xshift=9cm]
\coordinate (A) at (-3,3);
\coordinate (B) at (0,3);
\coordinate (C) at (0,0);
       \coordinate (D) at (-3,0) ;

\draw(A)node[above]{$\mathcal{K}_{1}$};
\draw(B)node[above]{$\mathcal{K}_{2}$};
\draw(C)node[below]{$\mathcal{K}_{3}$};
\draw(D)node[below]{$\mathcal{K}_{4}$};

    \fill (A) circle(1pt);
      \fill (B) circle(1pt);
        \fill (C) circle(1pt);
          \fill (D) circle(1pt);
        \draw  (A)--(B)--(C)--(A);  
\end{scope}
\end{tikzpicture}
    \end{ex}

\section{The number of quadratic residue graphs with $r$ vertices} \label{s6}

In this section, we calculate the number of quadratic residue graphs with $r$ vertices ($r \geq 1$). 
Let $\mathfrak{G}_r$ the set of all graphs with the vertex set $V = \{1, 2, \dots, r\}$, and $\mathfrak{C}_r$ be the subset of $\mathfrak{G}_r$ consisting of graphs whose connected components are all circuits. We recall by \cref{G-even-graph} that $\mathfrak{C}_r$ is the subset of $\mathfrak{G}_r$ consisting of even graphs. It is a crucial remark that, for instance, 
the graphs $G, G' \in \mathfrak{G}_3$, defined by 
$E(G) = \{\{1,2\}, \{1,3\}\}$, and $E(G') = \{\{1,2\}, \{2,3\}\}$, are isomorphic but are distinct in $\mathfrak{G}_3$.

\begin{lem} 
	We have $\#\mathfrak{G}_r=2^{\frac{r(r-1)}{2}}$.
\end{lem}

\begin{proof}
	For two distinct vertices $v_i$ and $v_j$, there are 2 possible ways according that $v_i$ and $v_j$ are connected by an edge or not.
Since there are $\binom{r}{2}=\frac{1}{2}r(r-1)$ ways for choosing 2 distinct vertices, the number of all graphs is $2^{\frac{r(r-1)}{2}}$.
\end{proof}

\begin{lem}\label{odd-even}
	For any graph $G=(V, E)$, the number of odd vertices is even.
\end{lem}

\begin{proof}
	Let $t:= \sum_{v\in{V}}{\deg(v)}$. Then we have $t=2\cdot \#E$ and so $t$ is even, from which the assertion follows.
\end{proof}

\begin{thm} \label{C/G-density}
	We have $\#\mathfrak{C}_r=2^{\frac{(r-1)(r-2)}{2}}$. Hence 
	\[
		\frac{\#\mathfrak{C}_r}{\#\mathfrak{G}_r}=\frac{1}{2^{r-1}}.
	\]
\end{thm}

\begin{proof}
	For $r=1$, we have $\# \mathfrak{C}_1 = 1$. For any $r \geq 2$, we define a map $\mathfrak{G}_{r-1} \to \mathfrak{C}_r$ as follows. For each $G \in \mathfrak{G}_{r-1}$, we add a new vertex connecting with all odd vertices of $G$. By \cref{odd-even}, the resulting graph is an even graph with $r$ vertices. We easily show that the map is bijective, and then, we have $\# \mathfrak{C}_r = \#\mathfrak{G}_{r-1} = 2^{\frac{(r-1)(r-2)}{2}}$.
\end{proof}

\section{A density formula for mod 2 arithmetic Dijkgraaf--Witten invariants} \label{s7}

In this section, we propose a density formula for mod 2 arithmetic Dijkgraaf--Witten invariants, based on the properties of quadratic residue graphs. First, we prepare some notations. Let $\mathcal{P}$ be the set of all prime numbers and $\mathcal{P}_{\equiv b (a)}$ the subset defined by $\mathcal{P}_{\equiv b (a)} = \{p \in \mathcal{P} \mid p \equiv b \pmod{a}\}$. For a positive integer $r > 0$, we define $P_r(x) = \{\{p_1, \dots, p_r\} \subset \mathcal{P}_{\equiv 1 (4)} \cap [1,x] \mid p_i \neq p_j\ (i \neq j)\}$. Our objective is to calculate the natural density of mod 2 arithmetic Dijkgraaf--Witten invariants defined by
\[
	d_r := \lim_{x \to \infty} \frac{\#\{\{p_1, \dots, p_r\} \in P_r(x) \mid Z_k = 2^{r-2}\}}{\# P_r(x)}.
\]
Let $S_r(x) = \{\{p_1,\dots, p_r\} \in P_r(x) \mid \text{any connected component of $G(\{p_1,\dots, p_r\})$ is a circuit}\}$, where $G(\{p_1,\dots, p_r\})$ is the quadratic residue graph associated to $\{p_1, \dots, p_r\}$. According to \cref{G-even-graph} and \cref{maint}, the above density is also expressed as $d_r = \lim_{x \to \infty} \#S_r(x)/\# P_r(x)$. Since we have
\[
	\prod_{1 \le j \le r} \bigg(1 + \prod_{\substack{1 \le i \le r \\ i \neq j}} \left(\frac{p_i}{p_j}\right)\bigg) = \begin{cases}
		2^r &\text{if } G(\{p_1, \dots, p_r\}) \text{ is an even graph},\\
		0 &\text{if otherwise},
	\end{cases}
\]
the number $\# S_r(x)$ is given by
\[
	\# S_r(x) = \frac{1}{2^r} \sum_{\{p_1, \dots, p_r\} \in P_r(x)} \prod_{1 \le j \le r} \bigg(1 + \prod_{\substack{1 \le i \le r \\ i \neq j}} \left(\frac{p_i}{p_j}\right)\bigg).
\]

\begin{thm}\label{density-theorem}
	The density $d_r$ exists and we have
	\[
		d_r = \frac{1}{2^{r-1}} = \frac{\# \mathfrak{C}_r}{\# \mathfrak{G}_r}.
	\]
\end{thm}

To prove this theorem, we begin by rephrasing the above counting formula for $\#S_r(x)$.

\begin{lem}
	For any positive integer $r > 0$, we have
	\[
		\#S_r(x) = \frac{1}{2^r} \frac{1}{r!} \sum_{\substack{(p_1, \dots, p_r) \in (\mathcal{P}_{\equiv 1 (4)} \cap [1,x])^r \\ p_i \neq p_j\ (i \neq j)}} \prod_{1 \le j \le r} \bigg(1 + \prod_{\substack{1 \le i \le r \\ i \neq j}} \left(\frac{p_i}{p_j}\right)\bigg).
	\]
\end{lem}

\begin{proof}
	The action of the symmetric group $\mathfrak{S}_r$ to the set
	\begin{align}\label{Qr(x)}
		Q_r(x) := \{(p_1, \dots, p_r) \in (\mathcal{P}_{\equiv 1 (4)} \cap [1,x])^r \mid p_i \neq p_j\ (i \neq j)\}
	\end{align}
	implies a natural $r!$ to $1$ correspondence between $Q_r(x)$ and $P_r(x)$, which immediately gives the desired equation.
\end{proof}

Furthermore, the right-hand side can be expanded as follows.
\begin{align*}
	\#S_r(x) &= \frac{1}{2^r r!} \sum_{(p_1, \dots, p_r) \in Q_r(x)} \left(2 + \sum_{\substack{e_1, \dots, e_r \in \{0, 1\} \\ 0 < e_1 + \cdots + e_r < r}} \left(\frac{\prod_{1 \le i \le r} p_i^{e_i}}{\prod_{1 \le i \le r} p_i^{1-e_i}} \right) \right)\\
	&= \frac{\# Q_r(x)}{2^{r-1} r!} + \frac{1}{2^r r!} \sum_{\substack{e_1, \dots, e_r \in \{0, 1\} \\ 0 < e_1 + \cdots + e_r < r}} \sum_{(p_1, \dots, p_r) \in Q_r(x)} \left(\frac{\prod_{1 \le i \le r} p_i^{e_i}}{\prod_{1 \le i \le r} p_i^{1-e_i}} \right),
\end{align*}
where $\left(\frac{\cdot}{\cdot}\right)$ is the Jacobi symbol.
To evaluate each term, we recall the following classical result (cf.~\cite[Section 7-4]{LeV}).

\begin{prop}[{The prime number theorem for arithmetic progressions}]
	For positive integers $a$ and $b$ with $(a,b) = 1$, the number $\pi_{a,b}(x) = \# (\mathcal{P}_{\equiv b (a)} \cap [1,x])$ satisfies
	\[
		\pi_{a,b}(x) \sim \frac{1}{\varphi(a)} \frac{x}{\log x}
	\]
	as $x \to \infty$, where $\varphi(a)$ is the Euler totient function.
\end{prop}

The estimation of the term $\# Q_r(x)$ immediately follows from the proposition. In fact, we have
\begin{align*}
	\# Q_r(x) = \prod_{j=0}^{r-1} (\pi_{4,1}(x) - j) \sim \pi_{4,1}(x)^r \sim \frac{1}{2^r} \frac{x^r}{\log^r x}
\end{align*}
as $x \to \infty$. To evaluate the remaining terms, a more detailed evaluation on Jacobi symbols is required.

\begin{thm}[{\cite[Theorem 1]{Heath}}]\label{HB}
	For any $\varepsilon > 0$, there exists a constant $C_\varepsilon > 0$ such that the following holds. Let $M, N$ be positive integers, and let $a_1, \dots, a_N$ be arbitrary complex numbers. Then 
	\[
		\sideset{}{^*}\sum_{m \le M} \left|\sideset{}{^*}\sum_{n \le N} a_n \left(\frac{n}{m}\right) \right|^2 \leq C_\varepsilon (MN)^\varepsilon (M+N) \sideset{}{^*}\sum_{n \le N} |a_n|^2,
	\]
	where $\sum\nolimits^*$ indicates restriction to positive odd square-free values and $\left(\frac{\cdot}{\cdot}\right)$ stands for the Jacobi symbol.
\end{thm}

By symmetric properties, all we need to estimate are the following cases.
\begin{lem}
	For any positive integer $0 < s < r$, we have
	\[
		E_s(x) = \sum_{(p_1, \dots, p_r) \in Q_r(x)} \left(\frac{p_{s+1} \cdots p_r}{p_1 \cdots p_s} \right) = o_r(\pi_{4,1}(x)^r)
	\]
	as $x \to \infty$.
\end{lem}

\begin{proof}
	We define
	\begin{align*}
		a_m^{(s)} (x)&= \begin{cases}
			s! &\text{if } m \text{ is a product of $s$ distinct primes in } \mathcal{P}_{\equiv 1(4)} \cap [1,x],\\
			0 &\text{if otherwise}
		\end{cases}
	\end{align*}
	and
	\begin{align*}
		b_m(x) = \begin{cases}
			\displaystyle{\sum_{\substack{(p_{s+1}, \dots, p_r) \in (\mathcal{P}_{\equiv 1(4)} \cap [1,x])^{r-s} \\ p_i \neq p_j\ (i \neq j)}} \left(\frac{p_{s+1} \cdots p_r}{m}\right)} &\text{if } m \text{ is square-free and odd},\\
			0 &\text{it otherwise}.
		\end{cases}
	\end{align*}
	Then we have $E_s(x) = \sum_{m \le x^s} a_m^{(s)}(x) b_m(x)$, where we change the variables via $m = p_1 \cdots p_s$. By the Cauchy--Schwarz inequality, 
	\begin{align*}
		E_s(x)^2 &\le \left(\sum_{m \le x^s} a_m^{(s)}(x)^2 \right) \cdot \left(\sum_{m \le x^s} b_m(x)^2 \right)\\
		&\le (s!)^2 \pi_{4,1}(x)^s \sideset{}{^*}\sum_{m \le x^s} \Bigg|\sum_{\substack{(p_{s+1}, \dots, p_r) \in (\mathcal{P}_{\equiv 1(4)} \cap [1,x])^{r-s} \\ p_i \neq p_j\ (i \neq j)}} \left(\frac{p_{s+1} \cdots p_r}{m}\right) \Bigg|^2\\
		&= (s!)^2 \pi_{4,1}(x)^s \sideset{}{^*}\sum_{m \le x^s} \bigg|\ \sideset{}{^*}\sum_{n \le x^{r-s}} a_n^{(r-s)}(x) \left(\frac{n}{m} \right) \bigg|^2,
	\end{align*}
	where we change the variables via $n = p_{s+1} \cdots p_r$. Therefore, by applying \cref{HB}, we obtain
	\begin{align*}
		E_s(x)^2 &\ll_{r,\varepsilon} \pi_{4,1}(x)^s x^{r\varepsilon} (x^s + x^{r-s}) \sideset{}{^*}\sum_{n \le x^{r-s}} |a_n^{(r-s)}(x)|^2\\
		&\ll_{r,\varepsilon} \pi_{4,1}(x)^s x^{r\varepsilon + r-1} \pi_{4,1}(x)^{r-s},
	\end{align*}
	which implies $E_s(x) \ll_{r,\varepsilon} x^{(r\varepsilon + r-1)/2} \pi_{4,1}(x)^{r/2}$. By taking $0 < \varepsilon < 1/2r$, we have $E_s(x) = o_r(\pi_{4,1}(x)^r)$.
\end{proof}

In conclusion, we have
\[
	\#S_r(x) = \frac{\# Q_r(x)}{2^{r-1} r!} + o_r(\pi_{4,1}(x)^r) \sim \frac{\pi_{4,1}(x)^r}{2^{r-1} r!}
\]
as $x \to \infty$. Since $\#P_r(x) = {\pi_{4,1}(x) \choose r} \sim \pi_{4,1}(x)^r/r!$, we conclude the proof of \cref{density-theorem}.


\begin{thebibliography}{Kim20}

\bibitem[Bol98]{boll}
B.~Bollob\'{a}s, \emph{Modern graph theory}, Graduate Texts in Mathematics,
  vol. 184, Springer-Verlag, New York, 1998.

\bibitem[DW90]{Dijkgraaf}
R.~Dijkgraaf and E.~Witten, \emph{Topological gauge theories and group
  cohomology}, Comm. Math. Phys. \textbf{129} (1990), no.~2, 393--429.

\bibitem[HB95]{Heath}
D.~R. Heath-Brown, \emph{A mean value estimate for real character sums}, Acta
  Arith. \textbf{72} (1995), no.~3, 235--275.

\bibitem[Hir23]{hirano}
H.~Hirano, \emph{On mod $2$ arithmetic {D}ijkgraaf--{W}itten invariants for
  certain real quadratic number fields}, to appear in Osaka J. Math., \textbf{60} no.4, 2023.

\bibitem[Kim20]{M}
M.~Kim, \emph{Arithmetic {C}hern-{S}imons theory {I}}, Galois covers,
  {G}rothendieck-{T}eichm\"{u}ller {T}heory and {D}essins d'{E}nfants, Springer
  Proc. Math. Stat., vol. 330, Springer, Cham, [2020] \copyright 2020,
  pp.~155--180.

\bibitem[LeV56]{LeV}
W.~J. LeVeque, \emph{Topics in number theory. {V}ol. $2$}, Addison-Wesley
  Publishing Co., Inc., Reading, Mass., 1956.

\bibitem[LP23]{Lee}
J.~Lee and J.~Park, \emph{Arithmetic {C}hern-{S}imons theory with real places},
  J. Knot Theory Ramifications \textbf{32} (2023), no.~4, Paper No. 2350027,
  25.

\bibitem[Maz73]{Mazur}
B.~Mazur, \emph{Notes on \'{e}tale cohomology of number fields}, Ann. Sci.
  \'{E}cole Norm. Sup. (4) \textbf{6} (1973), 521--552 (1974).

\bibitem[Mor12]{Morishita}
M.~Morishita, \emph{Knots and primes}, Universitext, Springer, London, 2012, An
  introduction to arithmetic topology.

\bibitem[Ono90]{ono}
T.~Ono, \emph{An introduction to algebraic number theory}, The University
  Series in Mathematics, Plenum Press, New York, 1990, Translated from the
  second Japanese edition by the author.

\end{thebibliography}


\newcommand{\noopsort}[1]{}
\providecommand{\bysame}{\leavevmode\hbox to3em{\hrulefill}\thinspace}
\providecommand{\MR}{\relax\ifhmode\unskip\space\fi MR }
\providecommand{\MRhref}[2]{%
  \href{http://www.ams.org/mathscinet-getitem?mr=#1}{#2}
}
\providecommand{\href}[2]{#2}



\newpage
\appendix 
\section{Appendix: Some numerical examples of $Z_k$} \label{appendix}

We give numerical computer calculations for Hirano's formula,
\[
	Z_k=\frac{1}{2}\sum_{\rho\in{{\rm Hom}(T,\mathbb{Z}/2\mathbb{Z})}}{\left(\prod_{i<j}{\left({\frac{p_i}{p_j}}\right)}^{\rho(b_{ij})}\right)}.
\]




\subsection*{The case that $r=3$}

We define $\rho_{0}$, $\rho_{1}$, $\rho_{2}$, and $\rho_{3}$ in $\mathrm{Hom}(T, \mathbb{Z} / 2 \mathbb{Z})$ by
\begin{table}[H]
\centering
\begin{tabular}{c|cccccccccc}
	$x$ & $(0,0,0)$ & $(1,1,0)$ & $(0,1,1)$ & $(1,0,1)$ \\ \hline
	$\rho_0(x)$ & 0 & 0 & 0 & 0 \\ 
	$\rho_1(x)$ & 0 & 1 & 0 & 1 \\
	$\rho_2(x)$ & 0 & 0 & 1 & 1 \\ 
	$\rho_3(x)$ & 0 & 1 & 1 & 0 \\ 
\end{tabular}
\end{table}

\begin{center}
 \begin{tabular}{|c|c|c|c|}
         \hline
   A set of prime numbers $S$ & Quadratic residue graph & $Z_k$\\
     \hline
5	,	13	,	17	&		
\begin{tikzpicture}[>=Stealth,scale=0.6]
   \coordinate (a) at (1.5,0);
  \coordinate (b) at (-0.75,1.299);
    \coordinate (c) at (-0.75,-1.299);
  \node  at (a)[anchor=225] {\small{\rm5}} ;
\node at (b)[anchor=-45] {\small{\rm13}} ;
\node at (c) [anchor=45]{\small{\rm17}} ;
    \draw[dashed] (0,0) circle(1.5);
     \draw[dashed,->](0,-1.7)--(0,1.7);
          \draw[dashed,->](-2,0)--(2,0);
         \fill (a) circle (2pt);
                \fill (b) circle (2pt);
                       \fill (c) circle (2pt);
                \draw [line width=1pt] (a)--(b);
                  \draw [line width=1pt] (a)--(c);
 \end{tikzpicture}		&	0	\\
 \hline
5	,	13	,	37	&
\begin{tikzpicture}[>=Stealth,scale=0.6]
   \coordinate (a) at (1.5,0);
  \coordinate (b) at (-0.75,1.299);
    \coordinate (c) at (-0.75,-1.299);
  \node  at (a)[anchor=225] {\small{\rm5}} ;
\node at (b)[anchor=-45] {\small{\rm13}} ;
\node at (c) [anchor=45]{\small{\rm37}} ;
    \draw[dashed] (0,0) circle(1.5);
     \draw[dashed,->](0,-1.7)--(0,1.7);
          \draw[dashed,->](-2,0)--(2,0);
         \fill (a) circle (2pt);
                \fill (b) circle (2pt);
                       \fill (c) circle (2pt);
                \draw [line width=1pt] (a)--(b)--(c)--(a);
 \end{tikzpicture}	&	2	\\
 \hline
5	,	13	,	41	&		\begin{tikzpicture}[>=Stealth,scale=0.6]
   \coordinate (a) at (1.5,0);
  \coordinate (b) at (-0.75,1.299);
    \coordinate (c) at (-0.75,-1.299);
  \node  at (a)[anchor=225] {\small{\rm5}} ;
\node at (b)[anchor=-45] {\small{\rm13}} ;
\node at (c) [anchor=45]{\small{\rm41}} ;
    \draw[dashed] (0,0) circle(1.5);
     \draw[dashed,->](0,-1.7)--(0,1.7);
          \draw[dashed,->](-2,0)--(2,0);
         \fill (a) circle (2pt);
                \fill (b) circle (2pt);
                       \fill (c) circle (2pt);
                \draw [line width=1pt] (a)--(b)--(c);
 \end{tikzpicture}	   &  	0	\\
 \hline
5	,	13	,	61	&		\begin{tikzpicture}[>=Stealth,scale=0.6]
   \coordinate (a) at (1.5,0);
  \coordinate (b) at (-0.75,1.299);
    \coordinate (c) at (-0.75,-1.299);
  \node  at (a)[anchor=225] {\small{\rm5}} ;
\node at (b)[anchor=-45] {\small{\rm13}} ;
\node at (c) [anchor=45]{\small{\rm61}} ;
    \draw[dashed] (0,0) circle(1.5);
          \draw[dashed,->](0,-1.7)--(0,1.7);
          \draw[dashed,->](-2,0)--(2,0);
         \fill (a) circle (2pt);
                \fill (b) circle (2pt);
                       \fill (c) circle (2pt);
                \draw [line width=1pt] (a)--(b);
                 
 \end{tikzpicture}		&	0	\\
 \hline
5	,	29	,	37		&	\begin{tikzpicture}[>=Stealth,scale=0.6]
   \coordinate (a) at (1.5,0);
  \coordinate (b) at (-0.75,1.299);
    \coordinate (c) at (-0.75,-1.299);
  \node  at (a)[anchor=225] {\small{\rm5}} ;
\node at (b)[anchor=-45] {\small{\rm29}} ;
\node at (c) [anchor=45]{\small{\rm37}} ;
    \draw[dashed] (0,0) circle(1.5);
      \draw[dashed,->](0,-1.7)--(0,1.7);
          \draw[dashed,->](-2,0)--(2,0);
         \fill (a) circle (2pt);
                \fill (b) circle (2pt);
                       \fill (c) circle (2pt);
                \draw [line width=1pt] (b)--(c)--(a);
             
 \end{tikzpicture}		&	0	\\
 \hline

     \end{tabular}
\end{center}

\subsection*{The case that $r=4$}

We define $\rho_{0}, \rho_{1}, \dots, \rho_{7}$ in $\mathrm{Hom} (T, \mathbb{Z} / 2 \mathbb{Z})$ by
\begin{table}[H]
\centering
\scalebox{0.9}[0.9]{
\begin{tabular}{c|cccccccccc}
	$x$ & $(0,0,0,0)$ & $(1,1,0,0)$ & $(1,0,1,0)$ & $(1,0,0,1)$ & $(0,1,1,0)$ & $(0,1,0,1)$ & $(0,0,1,1)$ & $(1,1,1,1)$ \\ \hline
	$\rho_0(x)$ & 0 & 0 & 0 & 0 & 0 & 0 & 0 & 0\\ 
	$\rho_1(x)$ & 0 & 1 & 0 & 0 & 1 & 1 & 0 & 1\\ 
	$\rho_2(x)$ & 0 & 0 & 1 & 0 & 1 & 0 & 1 & 1\\ 
	$\rho_3(x)$ & 0 & 0 & 0 & 1 & 0 & 1 & 1 & 1\\ 
	$\rho_4(x)$ & 0 & 1 & 1 & 0 & 0 & 1 & 1 & 0\\ 
	$\rho_5(x)$ & 0 & 1 & 0 & 1 & 1 & 0 & 1 & 0\\ 
	$\rho_6(x)$ & 0 & 0 & 1 & 1 & 1 & 1 & 0 & 0\\ 
	$\rho_7(x)$ & 0 & 1 & 1 & 1 & 0 & 0 & 0 & 1\\ 
\end{tabular}
}
\end{table}

\begin{center}
       \begin{tabular}{|c|c|c|c|}
         \hline
     
     A set of prime numbers $S$ & Quadratic residue graph & $Z_k$\\
  \hline
  5	,	13	,	17	,	29	&	\begin{tikzpicture}[>=Stealth,scale=0.6]
   \coordinate (a) at (1.5,0);
  \coordinate (b) at (0,1.5);
    \coordinate (c) at (-1.5,0);
  \coordinate (d) at (0,-1.5);
  \node  at (a)[anchor=225] {\small{\rm5}} ;
\node at (b)[anchor=-45] {\small{\rm13}} ;
\node at (c) [anchor=-45]{\small{\rm17}} ;
\node at (d) [anchor=45]{\small{\rm29}} ;
    \draw[dashed] (0,0) circle(1.5);
     \draw[dashed,->](0,-2)--(0,2);
          \draw[dashed,->](-2,0)--(2,0);      
         \fill (a) circle (1.3pt);
                \fill (b) circle (1.3pt);
                       \fill (c) circle (1.3pt);
                              \fill (d) circle (1.3pt);
                              \draw [line width=1pt] (b)--(a)--(c)--(d);
                 
 \end{tikzpicture}
	&	0	\\
  \hline
5	,	13	,	17	,	41	&	\begin{tikzpicture}[>=Stealth,scale=0.6]
   \coordinate (a) at (1.5,0);
  \coordinate (b) at (0,1.5);
    \coordinate (c) at (-1.5,0);
  \coordinate (d) at (0,-1.5);
  \node  at (a)[anchor=225] {\small{\rm5}} ;
\node at (b)[anchor=-45] {\small{\rm13}} ;
\node at (c) [anchor=-45]{\small{\rm17}} ;
\node at (d) [anchor=45]{\small{\rm41}} ;
    \draw[dashed] (0,0) circle(1.5);
     \draw[dashed,->](0,-2)--(0,2);
          \draw[dashed,->](-2,0)--(2,0);      
         \fill (a) circle (1.3pt);
                \fill (b) circle (1.3pt);
                       \fill (c) circle (1.3pt);
                              \fill (d) circle (1.3pt);
                              \draw [line width=1pt] (a)--(b)--(d)--(c)--(a);
                               
 \end{tikzpicture}
	&	4	\\
  \hline
5	,	13	,	29	,	53	&	\begin{tikzpicture}[>=Stealth,scale=0.6]
   \coordinate (a) at (1.5,0);
  \coordinate (b) at (0,1.5);
    \coordinate (c) at (-1.5,0);
  \coordinate (d) at (0,-1.5);
  \node  at (a)[anchor=225] {\small{\rm5}} ;
\node at (b)[anchor=-45] {\small{\rm13}} ;
\node at (c) [anchor=-45]{\small{\rm29}} ;
\node at (d) [anchor=45]{\small{\rm53}} ;
    \draw[dashed] (0,0) circle(1.5);
     \draw[dashed,->](0,-2)--(0,2);
          \draw[dashed,->](-2,0)--(2,0);      
         \fill (a) circle (1.3pt);
                \fill (b) circle (1.3pt);
                       \fill (c) circle (1.3pt);
                              \fill (d) circle (1.3pt);
                              \draw [line width=1pt] (a)--(b);
                                \draw  [line width=1pt](a)--(d);
 \end{tikzpicture}
	&	0	\\
  \hline
5	,	13	,	29	,	61	&	\begin{tikzpicture}[>=Stealth,scale=0.6]
   \coordinate (a) at (1.5,0);
  \coordinate (b) at (0,1.5);
    \coordinate (c) at (-1.5,0);
  \coordinate (d) at (0,-1.5);
  \node  at (a)[anchor=225] {\small{\rm5}} ;
\node at (b)[anchor=-45] {\small{\rm13}} ;
\node at (c) [anchor=-45]{\small{\rm29}} ;
\node at (d) [anchor=45]{\small{\rm61}} ;
    \draw[dashed] (0,0) circle(1.5);
     \draw[dashed,->](0,-2)--(0,2);
          \draw[dashed,->](-2,0)--(2,0);      
         \fill (a) circle (1.3pt);
                \fill (b) circle (1.3pt);
                       \fill (c) circle (1.3pt);
                              \fill (d) circle (1.3pt);
                              \draw [line width=1pt] (a)--(b);
                                \draw  [line width=1pt](c)--(d);
 \end{tikzpicture}
 &	0	\\
  \hline
5	,	13	,	37	,	101	&	\begin{tikzpicture}[>=Stealth,scale=0.6]
   \coordinate (a) at (1.5,0);
  \coordinate (b) at (0,1.5);
    \coordinate (c) at (-1.5,0);
  \coordinate (d) at (0,-1.5);
  \node  at (a)[anchor=225] {\small{\rm5}} ;
\node at (b)[anchor=-45] {\small{\rm13}} ;
\node at (c) [anchor=-45]{\small{\rm37}} ;
\node at (d) [anchor=45]{\small{\rm101}} ;
    \draw[dashed] (0,0) circle(1.5);
     \draw[dashed,->](0,-2)--(0,2);
          \draw[dashed,->](-2,0)--(2,0);      
         \fill (a) circle (1.3pt);
                \fill (b) circle (1.3pt);
                       \fill (c) circle (1.3pt);
                              \fill (d) circle (1.3pt);
                              \draw [line width=1pt] (a)--(b)--(c)--(a);
                           
 \end{tikzpicture}
	&	4	\\
      \hline
      5	,	17	,	29	,	37	&	\begin{tikzpicture}[>=Stealth,scale=0.6]
   \coordinate (a) at (1.5,0);
  \coordinate (b) at (0,1.5);
    \coordinate (c) at (-1.5,0);
  \coordinate (d) at (0,-1.5);
  \node  at (a)[anchor=225] {\small{\rm5}} ;
\node at (b)[anchor=-45] {\small{\rm17}} ;
\node at (c) [anchor=-45]{\small{\rm29}} ;
\node at (d) [anchor=45]{\small{\rm37}} ;
    \draw[dashed] (0,0) circle(1.5);
     \draw[dashed,->](0,-2)--(0,2);
          \draw[dashed,->](-2,0)--(2,0);      
         \fill (a) circle (1.3pt);
                \fill (b) circle (1.3pt);
                       \fill (c) circle (1.3pt);
                              \fill (d) circle (1.3pt);
                              \draw [line width=1pt] (a)--(b)--(c)--(d)--(a);
                                 \draw [line width=1pt] (b)--(d);
                           
 \end{tikzpicture}
	&	0	\\
      \hline

       \end{tabular}

       \begin{tabular}{|c|c|c|c|}

         \hline

     A set of prime numbers $S$ & Quadratic residue graph & $Z_k$\\
      \hline
  
5	,	13	,	37	,	113	&	\begin{tikzpicture}[>=Stealth,scale=0.6]
   \coordinate (a) at (1.5,0);
  \coordinate (b) at (0,1.5);
    \coordinate (c) at (-1.5,0);
  \coordinate (d) at (0,-1.5);
  \node  at (a)[anchor=225] {\small{\rm5}} ;
\node at (b)[anchor=-45] {\small{\rm13}} ;
\node at (c) [anchor=-45]{\small{\rm37}} ;
\node at (d) [anchor=45]{\small{\rm113}} ;
    \draw[dashed] (0,0) circle(1.5);
     \draw[dashed,->](0,-2)--(0,2);
          \draw[dashed,->](-2,0)--(2,0);      
         \fill (a) circle (1.3pt);
                \fill (b) circle (1.3pt);
                       \fill (c) circle (1.3pt);
                              \fill (d) circle (1.3pt);
                              \draw [line width=1pt] (a)--(b)--(c)--(d)--(a);
                               \draw [line width=1pt] (a)--(c);
 \end{tikzpicture}
	&	0	\\
\hline
13	,	17	,	29	,	97	&	\begin{tikzpicture}[>=Stealth,scale=0.6]
   \coordinate (a) at (1.5,0);
  \coordinate (b) at (0,1.5);
    \coordinate (c) at (-1.5,0);
  \coordinate (d) at (0,-1.5);
  \node  at (a)[anchor=225] {\small{\rm5}} ;
\node at (b)[anchor=-45] {\small{\rm17}} ;
\node at (c) [anchor=-45]{\small{\rm29}} ;
\node at (d) [anchor=45]{\small{\rm97}} ;
    \draw[dashed] (0,0) circle(1.5);
     \draw[dashed,->](0,-2)--(0,2);
          \draw[dashed,->](-2,0)--(2,0);      
         \fill (a) circle (1.3pt);
                \fill (b) circle (1.3pt);
                       \fill (c) circle (1.3pt);
                              \fill (d) circle (1.3pt);
                              \draw [line width=1pt] (b)--(c)--(d)--(b);
                               \draw [line width=1pt] (a)--(d);
 \end{tikzpicture}&	0	\\
     \hline     
      
5	,	17	,	41	,	53	&	\begin{tikzpicture}[>=Stealth,scale=0.6]
   \coordinate (a) at (1.5,0);
  \coordinate (b) at (0,1.5);
    \coordinate (c) at (-1.5,0);
  \coordinate (d) at (0,-1.5);
  \node  at (a)[anchor=225] {\small{\rm5}} ;
\node at (b)[anchor=-45] {\small{\rm17}} ;
\node at (c) [anchor=-45]{\small{\rm41}} ;
\node at (d) [anchor=45]{\small{\rm53}} ;
    \draw[dashed] (0,0) circle(1.5);
     \draw[dashed,->](0,-2)--(0,2);
          \draw[dashed,->](-2,0)--(2,0);      
         \fill (a) circle (1.3pt);
                \fill (b) circle (1.3pt);
                       \fill (c) circle (1.3pt);
                              \fill (d) circle (1.3pt);
                              \draw [line width=1pt] (a)--(b)--(c)--(d)--(a);
 \end{tikzpicture}&	4	\\
     \hline

      13	,	17	,	53	,	73	&	\begin{tikzpicture}[>=Stealth,scale=0.6]
   \coordinate (a) at (1.5,0);
  \coordinate (b) at (0,1.5);
    \coordinate (c) at (-1.5,0);
  \coordinate (d) at (0,-1.5);
  \node  at (a)[anchor=225] {\small{\rm5}} ;
\node at (b)[anchor=-45] {\small{\rm17}} ;
\node at (c) [anchor=-45]{\small{\rm41}} ;
\node at (d) [anchor=45]{\small{\rm53}} ;
    \draw[dashed] (0,0) circle(1.5);
     \draw[dashed,->](0,-2)--(0,2);
          \draw[dashed,->](-2,0)--(2,0);      
         \fill (a) circle (1.3pt);
                \fill (b) circle (1.3pt);
                       \fill (c) circle (1.3pt);
                              \fill (d) circle (1.3pt);
                              \draw [line width=1pt] (b)--(d)--(c);
                               \draw [line width=1pt] (a)--(d);
 \end{tikzpicture}&	0	\\
     \hline 
           5	,	17	,	37	,	113	&	\begin{tikzpicture}[>=Stealth,scale=0.6]
   \coordinate (a) at (1.5,0);
  \coordinate (b) at (0,1.5);
    \coordinate (c) at (-1.5,0);
  \coordinate (d) at (0,-1.5);
  \node  at (a)[anchor=225] {\small{\rm5}} ;
\node at (b)[anchor=-45] {\small{\rm17}} ;
\node at (c) [anchor=-45]{\small{\rm37}} ;
\node at (d) [anchor=45]{\small{\rm113}} ;
    \draw[dashed] (0,0) circle(1.5);
     \draw[dashed,->](0,-2)--(0,2);
          \draw[dashed,->](-2,0)--(2,0);      
         \fill (a) circle (1.3pt);
                \fill (b) circle (1.3pt);
                       \fill (c) circle (1.3pt);
                              \fill (d) circle (1.3pt);
                              \draw [line width=1pt] (a)--(b)--(c)--(d)--(a);
                               \draw [line width=1pt] (b)--(d);
                                 \draw [line width=1pt] (a)--(c);
 \end{tikzpicture}&	0	\\
     \hline 
                5	,	29	,	41	,	89	&	\begin{tikzpicture}[>=Stealth,scale=0.6]
   \coordinate (a) at (1.5,0);
  \coordinate (b) at (0,1.5);
    \coordinate (c) at (-1.5,0);
  \coordinate (d) at (0,-1.5);
  \node  at (a)[anchor=225] {\small{\rm5}} ;
\node at (b)[anchor=-45] {\small{\rm29}} ;
\node at (c) [anchor=-45]{\small{\rm41}} ;
\node at (d) [anchor=45]{\small{\rm89}} ;
    \draw[dashed] (0,0) circle(1.5);
     \draw[dashed,->](0,-2)--(0,2);
          \draw[dashed,->](-2,0)--(2,0);      
         \fill (a) circle (1.3pt);
                \fill (b) circle (1.3pt);
                       \fill (c) circle (1.3pt);
                              \fill (d) circle (1.3pt);
                              \draw [line width=1pt] (b)--(c)--(d)--(b);
                   
 \end{tikzpicture}&	4	\\
     \hline 
                     5	,	37	,	61	,	101	&	\begin{tikzpicture}[>=Stealth,scale=0.6]
   \coordinate (a) at (1.5,0);
  \coordinate (b) at (0,1.5);
    \coordinate (c) at (-1.5,0);
  \coordinate (d) at (0,-1.5);
  \node  at (a)[anchor=225] {\small{\rm5}} ;
\node at (b)[anchor=-45] {\small{\rm37}} ;
\node at (c) [anchor=-45]{\small{\rm61}} ;
\node at (d) [anchor=45]{\small{\rm101}} ;
    \draw[dashed] (0,0) circle(1.5);
     \draw[dashed,->](0,-2)--(0,2);
          \draw[dashed,->](-2,0)--(2,0);      
         \fill (a) circle (1.3pt);
                \fill (b) circle (1.3pt);
                       \fill (c) circle (1.3pt);
                              \fill (d) circle (1.3pt);
                              \draw [line width=1pt] (a)--(b)--(c)--(d);
                   
 \end{tikzpicture}&	0	\\
     \hline 
\end{tabular}
\end{center}

\newpage
\leavevmode\\
Yuqi Deng\\
Graduate School of Mathematics, Kyushu University\\
744, Motooka, Nishi-ku, Fukuoka, 819-0395, JAPAN\\
e-mail: deng.yuqi.608@s.kyushu-u.ac.jp
\\ 
\\
Riku Kurimaru\\
Graduate School of Mathematics, Kyushu University\\
744, Motooka, Nishi-ku, Fukuoka, 819-0395, JAPAN\\
e-mail: rikureal@icloud.com
\\ 
\\
Toshiki Matsusaka\\
Faculty of Mathematics, Kyushu University\\
744, Motooka, Nishi-ku, Fukuoka, 819-0395, JAPAN\\
e-mail: matsusaka@math.kyushu-u.ac.jp

\end{document}